\documentclass[12pt]{article}

\usepackage{t1enc}
\usepackage[latin1]{inputenc}
\usepackage[english]{babel}
\usepackage{amsmath,amsthm}
\usepackage{amsfonts}
\usepackage{latexsym}
\usepackage{graphicx}
\usepackage{txfonts}
\usepackage[natural]{xcolor}

\newtheorem{thm}{Theorem}
\newtheorem{lem}[thm]{Lemma}
\newtheorem{cor}[thm]{Corollary}
\newtheorem{conj}[thm]{Conjecture}
\newtheorem{prop}[thm]{Proposition}

\makeatletter

\newcommand{\Rmnum}[1]{\expandafter\@slowromancap\romannumeral #1@}
\makeatother

\begin{document}

\title{Drawing complete multipartite graphs on the plane with restrictions on crossings\thanks{Supported by the National Natural Science Foundation of China (No.\,11301410, 11201440, 11101243), the Natural Science Basic Research Plan in Shaanxi Province of China (No.\,2013JQ1002), the Specialized
Research Fund for the Doctoral Program of Higher Education (No.\,20130203120021), and the Fundamental Research Funds for the Central Universities (No.\,K5051370003, K5051370021).}}
\author{Xin Zhang\thanks{Email address: xzhang@xidian.edu.cn.}\\
{\small Department of Mathematics, Xidian University, Xi'an 710071, China}}

\maketitle

\begin{abstract}\baselineskip  0.6cm
We introduce the concept of NIC-planar graphs and present the full characterization of NIC-planar complete $k$-partite graphs.
\\[.2em]
\end{abstract}

\baselineskip  0.6cm

\section{Introduction}

All graphs considered in this paper are finite, simple and undirected. By $V(G)$, $E(G)$, $\delta(G)$ and $\Delta(G)$, we denote the vertex set, the edge set, the minimum degree and the maximum degree of a graph $G$, respectively. For a plane graph $G$, $F(G)$ denotes its face set. Set $v(G)=|V(G)|,e(G)=|E(G)|$ and $f(G)=|F(G)|$. The \emph{crossing number} of a graph $G$, denoted by $cr(G)$, is the minimum possible number of crossings with which
the graph can be drawn on the plane. All graph drawings here are optimal, that is, all intersecting edges intersect in a single point that arises from exactly four distinct vertices.
For other undefined notations, we refer the readers to \cite{Bondy.2008}.

A graph is \emph{1-planar} if it can be drawn on a plane so that each edge is crossed by at most one other edge. The concept of the 1-planarity was introduced by Ringel \cite{Ringel} when he considered the vertex-face coloring of planar graphs (corresponding to the vertex coloring of 1-planar graphs). In \cite{Ringel}, Ringel gave the first result on the coloring of 1-planar graphs: every 1-planar graph is 7-colorable. Almost two decades later, Borodin \cite{Borodin,Bo} improved this bound to 6 and showed the sharpness of the new bound. Recently, more and more papers on the coloring problems of 1-planar graphs appear (see the introduction in \cite{ZL} for details). However, compared to the well-established planar graphs, the class of 1-planar graphs is still litter explored.

Now we transfer our attention to the subclasses of 1-planar graphs. Can we find a class of graphs that lies between planar graphs and 1-planar graphs? To answer this question, we need some new notions. A 1-planar drawing is \emph{good} if it contains the minimum number of crossings, and normally, we would assume that every 1-planar drawing is good. We know that every crossing (note that it is not a real vertex) is generalized by two mutually crossed edges, thus for every crossing $c$ there exists a vertex set $S(c)$ of size four, where $S(c)$ consists of the end-vertices of the two edges that generalize $c$. We call $S(c)$ the \emph{cluster set} around $c$.

\begin{prop}
If $c_1$ and $c_2$ are two crossings in a good 1-planar drawing of a 1-planar graph $G$, then $|S(c_1)\cap S(c_2)|\leq 2$.
\end{prop}

\begin{proof}
Let $S(c_1)=\{v_1,v_2,v_3,v_4\}$ and let $v_1v_2$ crosses $v_3v_4$. If $|S(c_1)\cap S(c_2)|=4$, then either $v_1v_3$ crosses $v_2v_4$ or $v_1v_4$ crosses $v_2v_3$, which are impossible. If $|S(c_1)\cap S(c_2)|=3$, then let $S(c_2)=\{v_1,v_2,v_3,v_5\}$ with $v_5\neq v_4$. By the definition of 1-planarity, $v_1v_2$ cannot be the vertex that is incident with $c_2$, thus we assume, without loss of generality, that $v_1v_3$ crosses $v_2v_5$ at $c_2$. We redrawn the edge $v_1v_3$ so that $v_1,v_3$ and $c_1$ form a closed area that has no interior vertex. This can be done since $c_1$ is not a real vertex of $G$ and the edges $v_1v_2,v_3v_4$ cannot be crossed any more. After doing so, we reduce the number of crossings by one, which contradicts the fact that the current drawing is good. Therefore, $|S(c_1)\cap S(c_2)|\leq 2$.
\end{proof}

In view of the above proposition, we can naturally define two class of graphs. First, we consider the graph that satisfies $|S(c_1)\cap S(c_2)|=0$ for every two crossings $c_1$ and $c_2$. In fact, this class of graphs has already been investigated by Kr\'al and Stacho \cite{Kral} and Zhang \emph{et al.}\,\cite{ZL-CEJM,ZLY} since 2010 under the notion of plane graphs with independent crossings, or IC-planar graphs for short. In particular, Zhang and Liu \cite{ZL-CEJM} showed that $e(G)\leq 13v(G)/4-6$ for every IC-planar graph $G$ and thus every IC-planar graph $G$ contains a vertex of degree at most 6 (the bound 6 is sharp because of the existence of a 6-regular IC-planar graph (see Fig.\,1 of \cite{ZL-CEJM})). Second, there is a class of graphs that satisfies $|S(c_1)\cap S(c_2)|\leq 1$ for every two crossings $c_1$ and $c_2$. Actually, this class of graphs is just the one we are introducing in this paper. From now on, we call such graphs \emph{plane graphs with near independent crossings}, or \emph{NIC-planar graphs} for short. Let $G$ be an NIC-planar graph and $uv$ be an edge of $G$. If $uv$ is crossed by at least two other edges, that is, $uv$ is incident with at least two crossings $c_1$ and $c_2$, then it is easy to see that $|S(c_1)\cap S(c_2)|\geq 2$, a contradiction. This implies that $G$ is 1-planar. Therefore, the class of NIC-planar graphs lies between IC-planar graphs and 1-planar graphs.

We know that a graph is planar if and only if it contains no $K_5$-minors or $K_{3,3}$-minors, that is to say, the class of planar graphs is minor closed. However, the classes of IC-planar graphs, NIC-planar graphs and 1-planar graphs are not minor closed. Indeed, take a plane drawing of $G$ and then for every crossing $c$ at which $v_1v_2$ crosses $v_3v_4$, place one new 2-valent vertex on each of the segments $cv_1,cv_2,cv_3$ and $cv_4$ (note that this operation is iterative), we then obtain an IC-planar subdivision of $G$, and thus also an NIC-planar and a 1-planar subdivision. Therefore, for any graph $H$, there exists an IC-planar graph (so an NIC-planar graph and a 1-planar graph) that contains an $H$-minor. This fact brings us a big obstruction to recognize those superclasses of planar graphs. In fact, it has already been proved by Korzhik and Mohar \cite{KM} that testing the 1-planarity is NP-hard, and we guess the following:

\begin{conj}\label{conj}
Testing the IC-planarity and the NIC-planarity are NP-hard.
\end{conj}

Now, how can we determine whether a given graph is IC-planar or NIC-planar? There are some feasible ways, one of which is to show that it has a large number of edges. For example, if we can prove that $e(G)>13v(G)/4-6$, then $G$ is not IC-planar. For NIC-planar graphs, we can show that $e(G)\leq 18(v(G)-2)/5$, the proof of which is left to the next section, thus graphs satisfy $e(G)> 18(v(G)-2)/5$ are not NIC-planar graphs. Note that every subgraph of an IC-planar graph or an NIC-planar graph is IC-planar or NIC-planar, thus if we can show that $G$ contains a non-IC-planar small graph or a non-NIC-planar small graph then $G$ is non-IC-planar or non-NIC-planar. Therefore, looking for non-IC-planar graphs and non-NIC-planar graphs with small number of vertices and edges seems helpful.
In the next section, we are to present the full characterizations of NIC-planar and IC-planar complete $k$-partite graphs.

\section{Main results and their proofs}
\newcommand{\gx}{G^{\times}}

In this section, we always assume that every NIC-planar graph and IC-planar graph has been drawn on the plane so that its NIC-planarity or IC-planarity is satisfied and the number of crossings is as small as possible. The $associated$ $plane$ $graph$ $\gx$ of $G$ is the plane graph that is obtained from $G$ by turning all crossings of $G$ into new $4$-valent vertices. We call the new added $4$-valent vertices in $\gx$ \emph{false vertices} and the faces incident a false vertex \emph{false faces}. If a vertex or face in $\gx$ is not false, then we call it \emph{true}.

\begin{lem}\label{cross-lem}
If $G$ is an NIC-planar graph, then $cr(G)\leq v(G)-2-\frac{1}{2}f_T(\gx)$, where $f_T(\gx)$ is the number of true faces in the associated plane graph $\gx$ of any NIC-planar drawing of $G$.
\end{lem}

\begin{proof}
It is easy to see that $2e(\gx)=\sum_{f\in F(\gx)}d(f)\geq 3f(\gx)=3f_T(\gx)+3f_F(\gx)$, where $f_F(\gx)$ denote the number of false faces in $\gx$. Since $e(\gx)=v(\gx)+f(\gx)-2$ by Euler's formula, $f_F(\gx)=4cr(G)$ and $v(\gx)=v(G)+cr(G)$, we have $cr(G)\leq v(G)-2-\frac{1}{2}f_T(\gx)$.
\end{proof}

\begin{thm}\label{cross-thm}
$cr(G)\leq \frac{3}{5}(v(G)-2)$ for any NIC-planar graph $G$.
\end{thm}

\begin{proof}
Let $G$ be an NIC-planar drawing of the graph and let $c$ be a crossing at which $v_1v_2$ crosses $v_3v_4$. By the definition of NIC-planarity, $v_1v_4,v_2v_4,v_1v_3$ and $v_2v_3$ (if exist) are not crossed. Therefore, if one edge  mentioned above, say $v_1v_4$ for example, do not exist in $G$, then we add it to $G$ so that the closed area formed by $v_1,v_4$ and $c$ contains no interior vertices, and if such an edge exists, then we redraw it if necessary so that the closed area formed by $v_1,v_4$ and $c$ has no interior vertices. For every crossing in $G$ we do the above operation, then we obtain a new NIC-planar graph $G_1$. Triangulate the associated plane graph $G_1^*$ of $G_1$ and denote the resulted graph by $G_2^*$. It is easy to see that $G_2^*$ is an associated plane graph of an NIC-planar graph $G_2$.
For any crossing $c$ generalized by $v_1v_2$ crossing $v_3v_4$ in $G_2$, there are four true faces of degree 3 that are incident with one of the edges among $v_1v_4,v_2v_4,v_1v_3$ and $v_2v_3$ in $G_2^*$. Thus, $4cr(G_2)\leq 3f_T(G_2^*)$, where $f_T(G_2^*)$ is the number of true faces in $G_2^*$. By Lemma \ref{cross-lem}, $cr(G_2)\leq v(G_2)-2-\frac{1}{2}f_T(G_2^*)\leq v(G)-2-\frac{2}{3}cr(G_2)$, which implies
$cr(G)\leq \frac{3}{5}(v(G)-2)$, since $cr(G)=cr(G_2)$.
\end{proof}

\begin{thm}\label{edge-thm}
$e(G)\leq \frac{18}{5}(v(G)-2)$ for any NIC-planar graph $G$.
\end{thm}

\begin{proof}
If we remove one edge from every pair of mutually crossed edges, we obtain a plane graph. Therefore, $e(G)\leq 3v(G)-6+cr(G)$ and by Theorem \ref{cross-thm} the result follows.
\end{proof}

By Theorem \ref{edge-thm}, it is easy to conclude that every NIC-planar graph $G$ contains a vertex of degree at most 7. However, this is not a new result since it already holds for 1-planar graphs. Surprisedly, by discharging method we can prove the following better result.

\begin{thm}\label{6-degenerate}
Every NIC-planar graph contains a vertex of degree at most $6$.
\end{thm}

\begin{proof}
Let $G$ be a counterexample to it and let $\gx$ be the associated plane graph of an NIC-planar drawing of $G$. Note that $\delta(G)\geq 7$. Assign to each vertex $v\in V(\gx)$ an initial charge $c(v)=d(v)-4$ and each face $f\in F(\gx)$ an initial charge $c(f)=d(f)-4$. By Euler's formula, we have $\sum_{v\in V(\gx)}(d(v)-4)+\sum_{f\in F(\gx)}(d(f)-4)=-8$, thus $\sum_{x\in V(\gx)\cup F(\gx)}c(x)=-8$. Define the discharging rule as follows:
\begin{description}
  \item[Rule] Every vertex of degree at least 7 transfers $\frac{1}{2}$ or $\frac{1}{3}$ to each of its incident false or true faces of degree 3, respectively.
\end{description}
Let $c'$ be the final charge function after discharging. It is easy to see that $c'(f)\geq 0$ for every $f\in F(\gx)$, since every false or true face is incident with two or three vertices of degree at least 7, respectively. Let $v$ be a vertex of $\gx$. If $d(v)\geq 8$, then $c'(f)\geq d(v)-4-\frac{1}{2}d(v)\geq 0$. If $d(v)=7$ and $v$ is incident with at most 6 faces of degree 3, then $c'(v)\geq 7-4-6\times \frac{1}{2}=0$. If $d(v)=7$ and $v$ is incident with only faces of degree 3, then by the NIC-planarity of $G$, $v$ is incident with at most four false faces, thus $c'(v)\geq 7-4-4\times\frac{1}{2}-3\times\frac{1}{3}=0$. If $d(v)=4$, then $c'(v)=0$. Therefore, $\sum_{x\in V(\gx)\cup F(\gx)}c(x)=\sum_{x\in V(\gx)\cup F(\gx)}c'(x)\geq 0$, a contradiction.
\end{proof}

Theorem \ref{6-degenerate} generalized Zhang and Liu's result in \cite{ZL-CEJM}: every IC-planar graph contains a vertex of degree at most $6$. Since there is a 6-regular IC-planar graph (see Fig.\,1 of \cite{ZL-CEJM}) and every IC-planar graph is NIC-planar, the bound 6 in Theorem \ref{6-degenerate} is best possible.

Theorems \ref{cross-thm}, \ref{edge-thm} and \ref{6-degenerate} tells us that any NIC-planar graph has small crossing number, small number of edges and small minimum degree. In the following, we are to present the full characterizations of NIC-planar and non-NIC-planar complete $k$-partite graphs, which are helpful to recognize some non-NIC-planar graphs.

\begin{thm}\label{1-part}
The complete graph $K_n$ is NIC-planar if and only if $n\leq 5$.
\end{thm}

\begin{proof}
It is easy to see that $K_5$ is NIC-planar since $cr(K_5)=1$. For $K_n$ with $n\geq 6$, $e(K_n)=\frac{1}{2}n(n+1)>\frac{18}{5}(n-2)=\frac{18}{5}(v(K_n)-2)$, so by Theorem \ref{edge-thm} it is not NIC-planar.
\end{proof}

\begin{lem}{\rm \cite{Kleitman}}\label{cross-Kmn}
If $m\leq 6$, then $cr(K_{m,n})=\lfloor\frac{m}{2}\rfloor\lfloor\frac{m-1}{2}\rfloor\lfloor\frac{n}{2}\rfloor\lfloor\frac{n-1}{2}\rfloor$.
\end{lem}

In the following, we use $Z(m,n)$ denotes the right member of the equality in Lemma \ref{cross-Kmn}.

\begin{thm}\label{2-part}
The complete bipartite graph $K_{m,n}$ with $m\geq n$ is NIC-planar if and only if $n\leq 2$, or $n=3$ and $m\leq 4$.
\end{thm}

\begin{proof}
Since $K_{4,3}$ has an NIC-planar drawing (see Fig.\,\ref{K34}) and $K_{m,n}$ is planar if $n\leq 2$, the sufficiency holds. To prove the necessity, we just need to show that $K_{5,3}$ and $K_{4,4}$ are not NIC-planar, since any graph containing as a subgraph a non-NIC-planar graph is non-NIC-planar. If they are NIC-planar graphs, then by Theorem \ref{cross-thm}, $cr(K_{5,3})\leq 3$ and $cr(K_{4,4})\leq 3$. However, $cr(K_{5,3})=cr(K_{4,4})=4$ by Lemma \ref{cross-Kmn}, a contradiction.
\end{proof}

\begin{figure}
  \begin{center}
  \includegraphics[width=4cm,height=3.6cm]{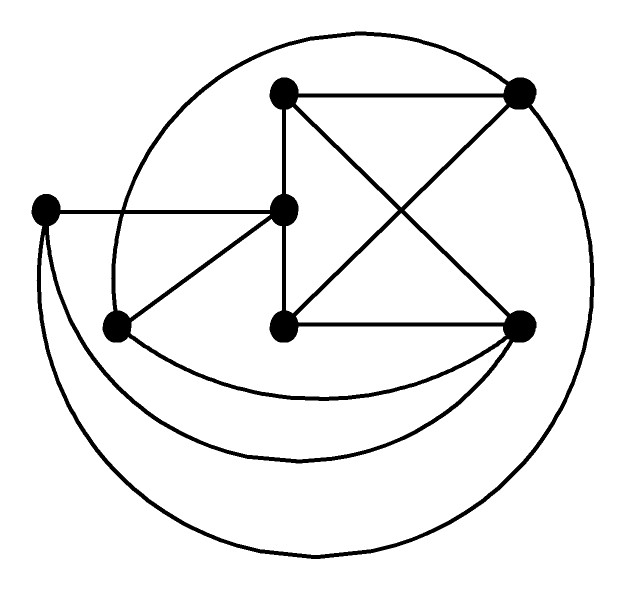}
\\\end{center}
  \caption{An NIC-planar drawing of $K_{4,3}$}\label{K34}
\end{figure}


\begin{lem}{\rm \cite{Asano}}\label{cross-K13n-K23n}
$cr(K_{1,3,n})=Z(4,n)+\lfloor\frac{n}{2}\rfloor$ and $cr(K_{2,3,n})=Z(5,n)+n$.
\end{lem}


\begin{lem}\label{cross-8-vertex}
If $G$ is NIC-planar and $v(G)\leq 8$, then $cr(G)\leq 2$.
\end{lem}

\begin{proof}
Without loss of generality, assume that $V(G)=\{v_1,\ldots,v_8\}$ and that there are three crossings $c_1,c_2,c_3$ in an NIC-planar drawing of $G$. Suppose that  $S(c_1)=\{v_1,v_2,v_3,v_4\}$ and $S(c_2)\supset\{v_5,v_6,v_7\}$. If $v_8\in S(c_2)$, then $|S(c_3)\cap S(c_i)|\geq 2$ for some $i\in \{1,2\}$, which contradicts the NIC-planarity of $G$. If $v_8\not\in S(c_2)$, then assume, without loss of generality, that $v_4\in S(c_2)$, which still implies that $|S(c_3)\cap S(c_i)|\geq 2$ for some $i\in \{1,2\}$, a contradiction.
\end{proof}

\begin{thm}\label{3-part}
The complete $3$-partite graph $K_{a_1,a_2,a_3}$ with $a_1\geq a_2\geq a_3$ is NIC-planar if and only if $a_2=1$, or $a_1\leq 4,a_2=2$ and $a_3=1$, or $a_1=a_2=a_3=2$.
\end{thm}

\begin{figure}
  \begin{center}
  \includegraphics[width=6cm,height=3.2cm]{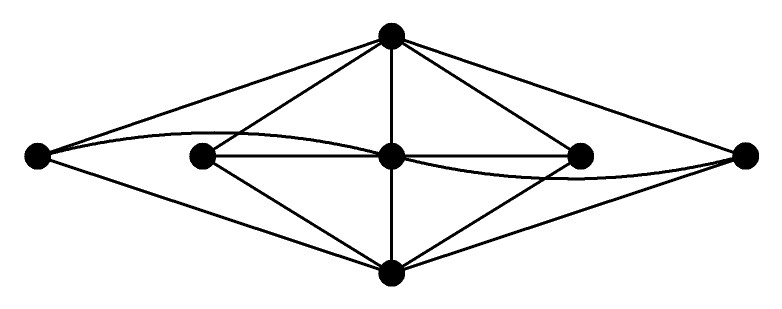}
\\\end{center}
  \caption{An NIC-planar drawing of $K_{4,2,1}$}\label{K421}
\end{figure}

\begin{proof}
Since $K_{4,2,1}$ has an NIC-planar drawing (see Fig.\,\ref{K421}) and $K_{2,2,2}$, $K_{a_1,1,1}$ are planar, the sufficiency holds. To prove the necessity, we just need to show that $K_{5,2,1}$, $K_{3,3,1}$ and $K_{3,2,2}$ are not NIC-planar graphs. Since $K_{5,2,1}$ contains as a subgraph $K_{5,3}$ which is non-NIC-planar by Theorem \ref{2-part}, it is non-NIC-planar.
If $K_{3,3,1}$ is NIC-planar, then Lemma \ref{cross-8-vertex} implies that $cr(K_{3,3,1})\leq 2$, which contradicts the fact that $cr(K_{3,3,1})=3$ by Lemma \ref{cross-K13n-K23n}. In the following, we claim that $K_{3,2,2}$ has no NIC-planar drawings.

If $K_{3,2,2}$ has an NIC-planar drawing $G$ with $c$ crossings, then the associated plane graph $\gx$ of $G$ has $16+2c$ edges, $11+c$ faces. $3+c$ vertices of degree 4 and four vertices of degree 5. On the other hand, we have $\sum_{v\in V(\gx)}(d(v)-4)+\sum_{f\in F(\gx)}(d(f)-4)=-8$, which implies that $\sum_{f\in F(\gx)}(d(f)-4)=-12$. Therefore, $\gx$ has at least 12 faces of degree 3. By the proof of Lemma \ref{cross-8-vertex}, we conclude that the NIC-drawing $G$ of the 7-vertex graph $K_{3,2,2}$ has at most two crossings, so $c\leq 2$. By Lemma \ref{cross-K13n-K23n}, we have $cr(3,2,2)=2$, thus $c=2$. Hence $\gx$ has 13 faces, twelve of which are of degree 3 and one of them is of degree 4. Note that there may be six types of crossings in $G$, see Fig.\,\ref{6}, where the vertices marked by $\alpha_i$ are taken from the $i$-th part of $K_{3,2,2}$.
  \begin{figure}
  \begin{center}
  \includegraphics[width=14cm,height=2cm]{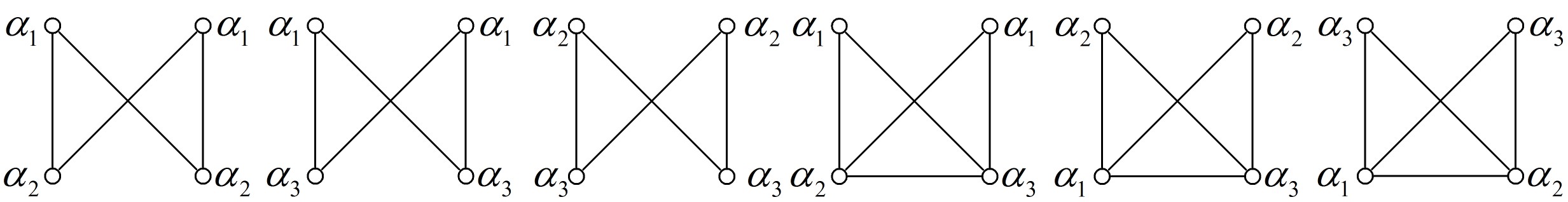}\label{6}
\\\end{center}
\caption{Six types of crossings in any possible NIC-drawing of $K_{3,2,2}$}
\end{figure}

It is easy to see that if there is one crossing point $u$ in $G$ as in one of the first three graphs in Fig.\,\ref{6}, then $u$ is a false vertex in $\gx$ that is incident with at least one face of degree at least 6 or at least two faces of degree ar least 4, a contradiction to the fact that $F(\gx)$ consists of 12 faces of degree 3 and one face of degree 4. If there is one crossing point $u$ in $G$ as in one of the last three graphs in Fig.\,\ref{6}, then $u$ is a false vertex in $\gx$ that is incident with exactly one face of degree at least 4, say $f_u$, and moreover, if $u$ and $v$ are two different crossing points of $G$ in this type and $d(f_u)=d(f_v)=4$, then $f_u\neq f_v$ by the NIC-planarity of $G$. This implies that there is a face of degree at least 5 or two faces of degree 4 in $\gx$, a contradiction.
\end{proof}


\begin{thm}\label{4-part}
The complete $4$-partite graph $K_{a_1,a_2,a_3,a_4}$ with $a_1\geq a_2\geq a_3\geq a_4$ is NIC-planar if and only if $a_1\leq 4$ and $a_2=1$, or $a_1=a_2=2$ and $a_3=1$.
\end{thm}

\begin{figure}
  \begin{center}
  \includegraphics[width=4.5cm,height=3.2cm]{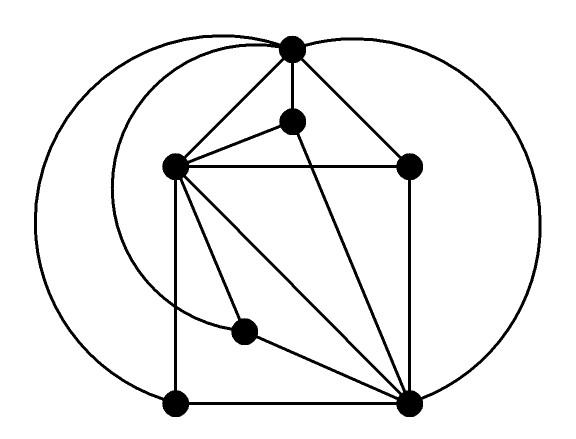}
\\\end{center}
  \caption{An NIC-planar drawing of $K_{4,1,1,1}$}\label{K4111}
\end{figure}

\begin{proof}
Since $K_{4,1,1,1}$ has an NIC-planar drawing (see Fig.\,\ref{K4111}) and $K_{2,2,1,1}$ has a drawing with only one crossing (thus has an NIC-planar drawing), the sufficiency holds. To prove the necessity, we just need to show that $K_{5,1,1,1}$, $K_{3,2,1,1}$ and $K_{2,2,2,2}$ are not NIC-planar graphs. Since $K_{5,1,1,1}$ and $K_{2,2,2,2}$ contains as a subgraph $K_{5,3}$ and $K_{4,4}$ which are non-NIC-planar by Theorem \ref{2-part}, respectively, they are non NIC-planar.
Since $K_{3,2,2}$ is non-NIC-planar by Theorem \ref{3-part} and it is a subgraph of $K_{3,2,1,1}$, $K_{3,2,1,1}$ is non-NIC-planar.
\end{proof}

\begin{lem}{\rm \cite{Ho2009}}\label{cross-K1111n}
$cr(K_{1,1,1,1,n})=Z(4,n)+n$.
\end{lem}

\begin{lem}\label{cross-6-vertex}
If $G$ is NIC-planar and $v(G)\leq 6$, then $cr(G)\leq 1$.
\end{lem}

\begin{proof}
If there are at least two crossings $c_1$ and $c_2$, then by the definition of NIC-planarity we have $|S(c_1)\cup S(c_2)|\geq 7$, contradicting the assumption $v(G)\leq 6$.
\end{proof}

\begin{thm}\label{5-part}
There is only one complete $5$-partite graph, that is, $K_{1,1,1,1,1}$.
\end{thm}

\begin{proof}
It is easy to see that $K_{1,1,1,1,1}=K_5$ has a drawing with only one crossing, thus it is NIC-planar. By Lemma \ref{cross-K1111n}, we have $cr(K_{2,1,1,1,1})=2$. However, if $K_{2,1,1,1,1}$ is NIC-planar, then by Lemma \ref{cross-6-vertex}, $cr(K_{2,1,1,1,1})\leq 1$, a contradiction.
\end{proof}

\begin{thm}\label{6-part}
There is no NIC-planar complete $t$-partite graphs with $t\geq 6$.
\end{thm}

\begin{proof}
Since any complete $t$-partite graph $G$ with $t\geq 6$ contains as a subgraph $K_6$ which is a non-NIC-planar graph by Theorem \ref{1-part}, $G$ is not NIC-planar.
\end{proof}


We now collect our results to the following table that presents the full characterization of NIC-planar complete $k$-partite graphs.
Note that the graphs after which we mark a star are planar graphs, and after which we mark two stars are non-planar IC-planar graphs.

\begin{center}
\begin{tabular}{|c|c|}
  \hline
  $k$ & NIC-planar complete $k$-partite graphs \\\hline\hline
  2 & $K_{1,n}$(*), $K_{2,n}$(*), $K_{3,3}$(**), $K_{3,4}$ \\
  3 & $K_{1,1,n}$(*), $K_{1,2,2}$(*), $K_{1,2,3}$(**), $K_{1,2,4}$ ,$K_{2,2,2}$(*) \\
  4 & $K_{1,1,1,1}$(*), $K_{1,1,1,2}$(*), $K_{1,1,1,3}$(**), $K_{1,1,1,4}$, $K_{1,1,2,2}$(**)  \\
  5 & $K_{1,1,1,1,1}$(**) \\
  \hline
\end{tabular}
\end{center}

Since IC-planar graphs is the subclass of NIC-planar graphs, we deduce the following full characterization of IC-planar complete $k$-partite graphs.

\begin{center}
\begin{tabular}{|c|c|}
  \hline
  $k$ & IC-planar complete $k$-partite graphs \\\hline\hline
  2 & $K_{1,n}$(*), $K_{2,n}$(*), $K_{3,3}$ \\
  3 & $K_{1,1,n}$(*), $K_{1,2,2}$(*), $K_{1,2,3}$, $K_{2,2,2}$(*) \\
  4 & $K_{1,1,1,1}$(*), $K_{1,1,1,2}$(*), $K_{1,1,1,3}$, $K_{1,1,2,2}$  \\
  5 & $K_{1,1,1,1,1}$ \\
  \hline
\end{tabular}
\end{center}


\begin{cor}\label{non-NIC-mul}
Any non-NIC-planar complete multipartite graphs contains one of the following graph: $K_{3,5}$, $K_{2,2,3}$, $K_{1,1,1,1,2}$.
\end{cor}

\begin{cor}\label{non-IC-mul}
Any non-IC-planar complete multipartite graph contains $K_{3,4}$ or $K_{1,1,1,1,2}$.
\end{cor}

As we have mentioned in Sec.1, if we are to prove that a graph $G$ is non-NIC-planar, we can try to prove that $G$ contains a known non-NIC-planar graph $H$. It is easy to see that $H$ has a subgraph $H'$ such that $H'-e$ is NIC-planar. We call such a graph $H'$ \emph{minimum non-NIC-planar graph}. Therefore, every non-NIC-planar graph contains a minimum non-NIC-planar graph. How many minimum non-NIC-planar graphs are they? If the number is finite, then a graph is non-NIC-planar if and only if it contains one of those minimum non-NIC-planar graphs. This seems to be a good characterization of NIC-planar graphs. Unluckily, the answer of the above question is in the negative.

\begin{lem}{\rm \cite{KPS}}\label{cross-Kn}
$cr(K_n)\geq 0.8594Z(n)$, where $Z(n)=\frac{1}{4}\lfloor\frac{n}{2}\rfloor\lfloor\frac{n-1}{2}\rfloor\lfloor\frac{n-2}{2}\rfloor\lfloor\frac{n-3}{2}\rfloor$.
\end{lem}

\begin{thm}\label{non-NIC}
There are infinite many minimum non-NIC-planar graphs.
\end{thm}

\begin{proof}
Suppose that there are finite many minimum non-NIC-planar graphs. Let $M$ be the maximum order of those minimum non-NIC-planar graphs. Since every 5-vertex graph is NIC-planar by Theorem \ref{1-part}, $M\geq 6$. We now place $M$ vertices of degree 2 on each edge of the complete graph $K_{4M}$. It is easy to see that the resulted graph $G$ is a graph with $v(G)=4M+2M(4M-1)M<2M(2M-1)(2M+1)<M(2M-1)^2(M-1)<\frac{5}{3}cr(K_{4M})=\frac{5}{3}cr(G)$ by Lemma \ref{cross-Kn}. Therefore, $G$ is non-NIC-planar by Theorem \ref{cross-thm}, and moreover, $G$ contains as a subgraph a minimum non-NIC-planar graph $H$. Since $H$ shall be 2-connected, $H$ contains at least one subdivided edge of $G$ in its entirety, which implies that $v(H)\geq M+2$. This contradicts the definition of $M$.
\end{proof}

By similar arguments as in the proof of Theorem \ref{non-NIC}, we can also prove the following. Note that the notion of minimum non-IC-planar graph can be defined similarly to the one of minimum non-NIC-planar graph. In fact, Theorems \ref{non-NIC} and \ref{nonIC} support Conjecture \ref{conj} in some sense.

\begin{thm}\label{nonIC}
There are infinite many minimum non-IC-planar graphs.
\end{thm}

Since $K_5$ has an IC-planar drawing on the plane with only crossing, every 5-vertex graph is IC-planar and thus is NIC-planar.

\begin{thm}\label{min}
$G:=K_6-K_2$ is non-NIC-planar and $H:=G-K_2$ is IC-planar.
\end{thm}

\begin{figure}
  \begin{center}
  \includegraphics[width=6.5cm,height=3.2cm]{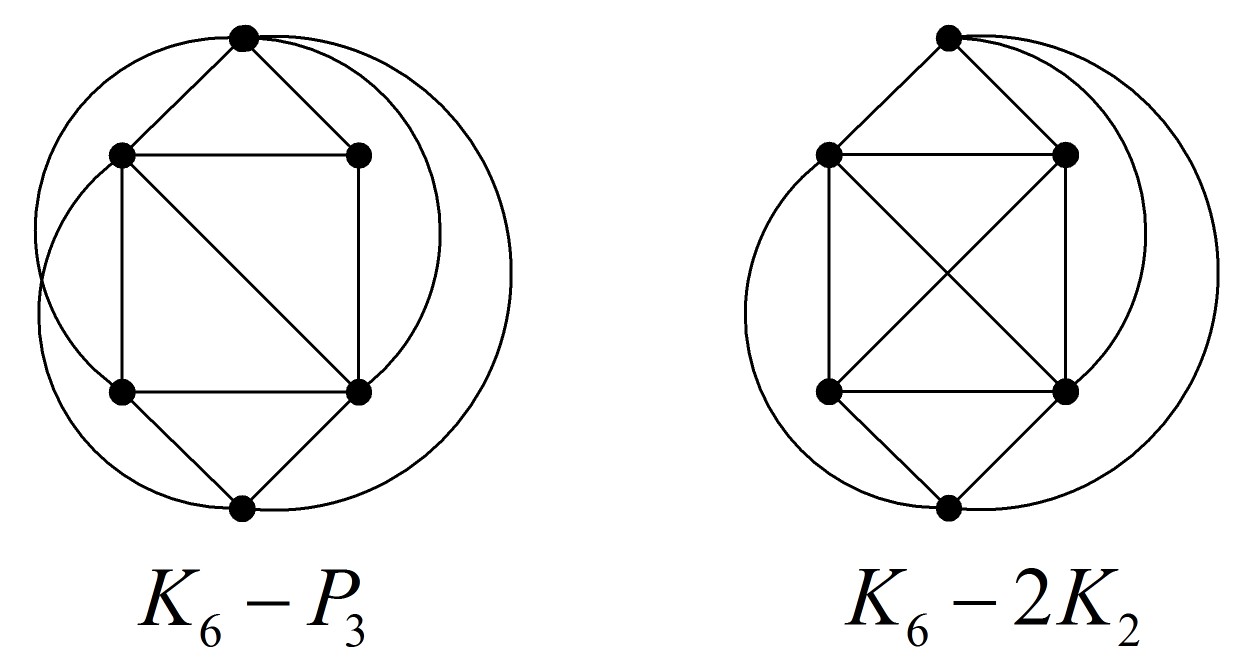}
\\\end{center}
  \caption{IC-planar drawings of $G-K_2$  }\label{K6-K2-K2}
\end{figure}

\begin{proof}
Since $cr(K_6)=Z(6)=3$ (see \cite{Guy}), $cr(K_6-K_2)\geq 2$. If $K_6-K_2$ is NIC-planar, then $cr(K_6-K_2)\leq 1$ by Lemma \ref{cross-6-vertex}, a contradiction. Therefore, $K_6-K_2$ is non-NIC-planar, and thus is non-IC-planar.
Since $G-K_2$ is either the graph $K_6-P_3$ or the graph $K_6-2K_2$ and Fig.\,\ref{K6-K2-K2} shows the IC-planar drawings of these two graphs, $G-K_2$ is IC-planar and is NIC-planar.
\end{proof}

By Theorem \ref{min}, we have the following immediate corollary.

\begin{cor}\label{K6-K2}
$K_6-K_2$ is the unique minimum non-NIC-planar $6$-vertex graph and is also the unique minimum non-IC-planar $6$-vertex graph.
\end{cor}

\begin{thm}\label{use-NIC}
$K_{3,5}$ and $K_{2,2,3}$ are both minimum non-NIC-planar graphs.
\end{thm}

\begin{figure}
  \begin{center}
  \includegraphics[width=9cm,height=3.5cm]{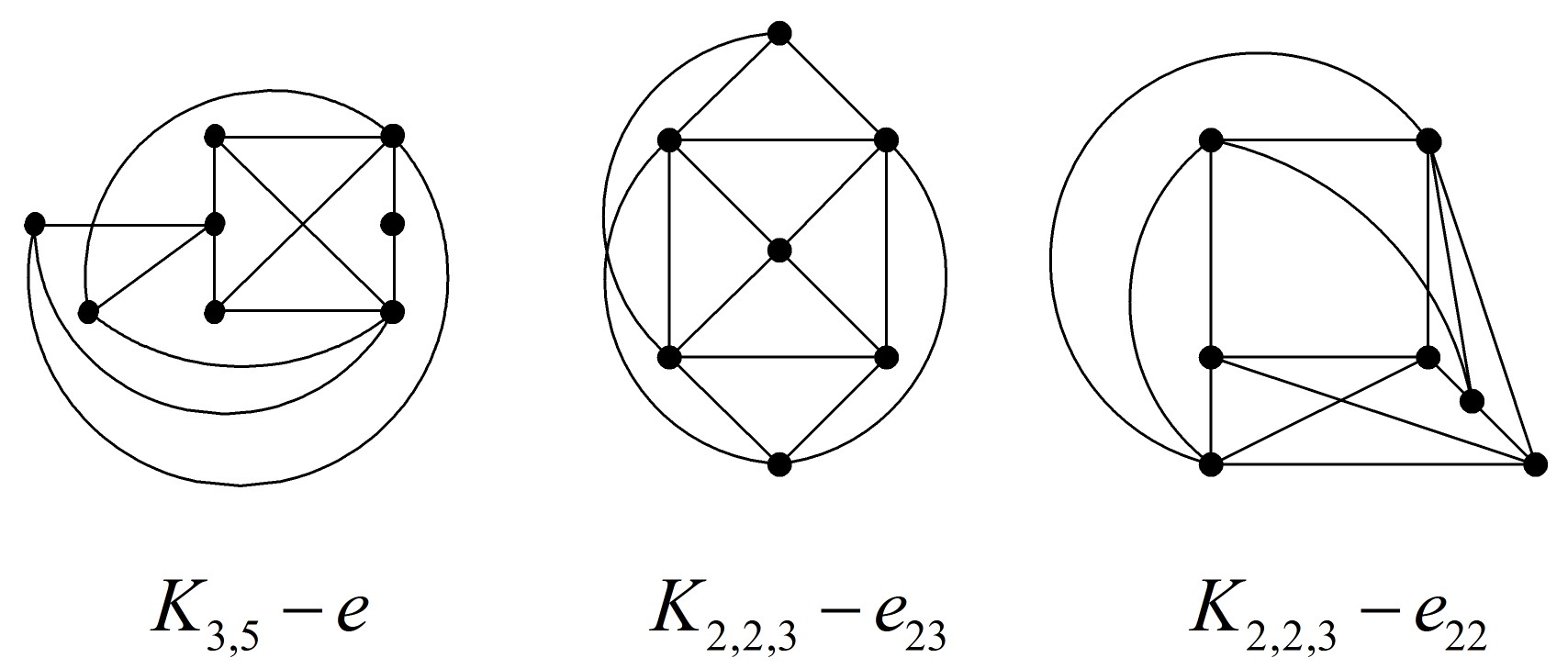}
\\\end{center}
  \caption{$K_{3,5}-e$ and $K_{2,2,3}-e$ are both NIC-planar}\label{min-non-NIC}
\end{figure}

\begin{proof}
First, $K_{3,5}$ and $K_{2,2,3}$ are non-NIC-planar by Theorems \ref{2-part} and \ref{3-part}. On the other hand, $K_{3,5}-e$ and $K_{2,2,3}-e$ are both NIC-planar, see Fig.\,\ref{min-non-NIC}. Note that $K_{2,2,3}-e$ would represent two graphs: one is obtained from $K_{2,2,3}$ by removing an edge $e_{22}$ between the two 2-vertex parts, and the other is obtained from $K_{2,2,3}$ by removing an edge $e_{23}$ between one 2-vertex part and one 3-vertex part.
\end{proof}

By Corollary \ref{non-NIC-mul}, Corollary \ref{K6-K2} and Theorem \ref{use-NIC}, we have the following result. Note that $K_6-K_2\cong K_{1,1,1,1,2}$.

\begin{cor}
There are only three minimum non-NIC-planar complete multipartite graphs: $K_{3,5}$, $K_{2,2,3}$ and $K_{1,1,1,1,2}$.
\end{cor}

\begin{figure}
  \begin{center}
  \includegraphics[width=3cm,height=3cm]{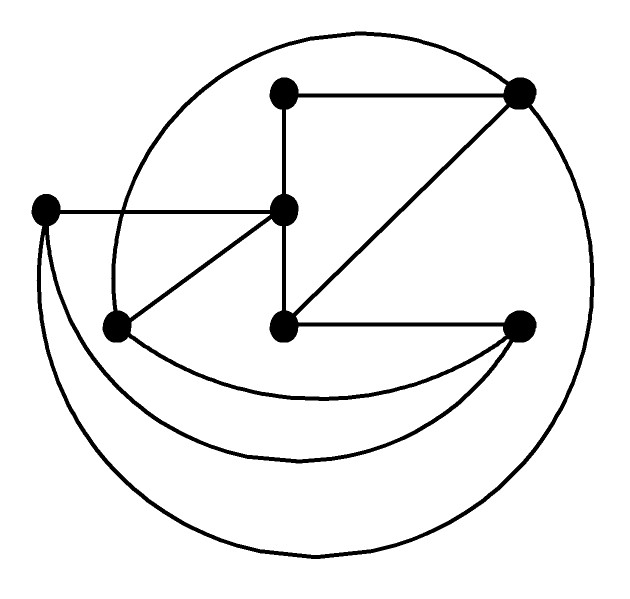}
\\\end{center}
  \caption{$K_{3,4}-e$ is IC-planar}\label{K34-e}
\end{figure}

For IC-planar complete multipartite graphs, we know that $K_{3,4}$ is non-IC-planar and $K_{3,4}-e$ has an IC-planar drawing (see Fig.\,\ref{K34-e}). Therefore,  $K_{3,4}$ is a minimum non-IC-planar graph. Hence, by Corollary \ref{non-IC-mul} and Corollary \ref{K6-K2}, we deduce the following.

\begin{cor}
There are only two minimum non-IC-planar complete multipartite graphs: $K_{3,4}$ and $K_{1,1,1,1,2}$.
\end{cor}

\end{document}